\documentclass[english, 12pt, twoside, a4paper]{article}

\usepackage{verbatim}
\usepackage{textcomp}
\usepackage[english]{babel}
\usepackage{amsmath}
\usepackage{amssymb}
\usepackage{amsfonts}
\usepackage{amsthm}
\usepackage{graphicx}

\newtheorem{theorem}{Theorem}
\newtheorem{prop}[theorem]{Proposition}

\theoremstyle{definition}
\newtheorem{definition}[theorem]{Definition}
\newtheorem{corollary}[theorem]{Corollary}
\newtheorem{example}[theorem]{Example}
\newtheorem{question}{Question}
\newtheorem{lemma}[theorem]{Lemma}

\newcommand{\defn}[1]{\textbf{#1}}

\begin{document}
\author{Nadav Kohen}
\title{Consecutive Radio Labeling of Hamming Graphs}
\date{}
\maketitle

\begin{abstract}
For a graph $G$, a $k$-radio labeling of $G$ is the assignment of positive integers to the vertices of $G$ such that the closer two vertices are on the graph, the greater the difference in labels is required to be. Specifically, $\vert f(u)-f(v)\vert\geq k + 1 - d(u,v)$ where $f(u)$ is the label on a vertex $u$ in $G$. Here, we consider the case when $G$ is the Cartesian products of complete graphs. Specifically we wish to find optimal labelings that use consecutive integers and determine when this is possible. We build off of a paper by Amanda Niedzialomski and construct a framework for discovering consecutive radio labelings for Hamming Graphs, starting with the smallest unknown graph, $K_3^4$, for which we provide an optimal labeling using our construction.
\end{abstract}

\section{Introduction}
Graph labeling was first introduced by Rosa in 1966 \cite{Rosa}. Since then, numerous types of labeling have been subject to extensive study including vertex coloring, graceful labeling, harmonious labeling, $k$-radio labeling, and more. For a survey of graph labeling, see Gallian \cite{Gallian}.

This paper will continue a portion of Niedzialomski's work \cite{Amanda} on radio labeling Hamming graphs (which have strong connections to coding theory, see \cite{ChangLuZhou} and \cite{Zhou}), paying some attention to graphs of the form $K_n^t$. These graphs are defined via the Cartesian Product.
\begin{definition}\label{cartproddef}
Given two simple connected graphs, $G$ and $H$, define the \defn{Cartesian Product}, $G\square H$, to have the vertex set $V(G)\times V(H)$ and edges such that a vertex $(v_i,u_i)\in G\square H$ is adjacent to $(v_j,u_j)\in G\square H$ if $v_i=v_j$ and $u_i$ is adjacent to $u_j$ in $H$ or if $u_i=u_j$ and $v_i$ is adjacent to $v_j$ in $G$ (See Figure \ref{Fig:cartprod}). Write the Cartesian Product of $t$ copies of a graph, $G$, as $G^t$. A \defn{Hamming graph} is a graph of the form $K_{n_1}^{t_1}\square K_{n_2}^{t_2}\square\cdots\square K_{n_m}^{t_m}$, where $K_i$ is the complete graph on $i$ vertices.
\end{definition}
\begin{figure}[h]
\centering
\includegraphics[scale=.19]{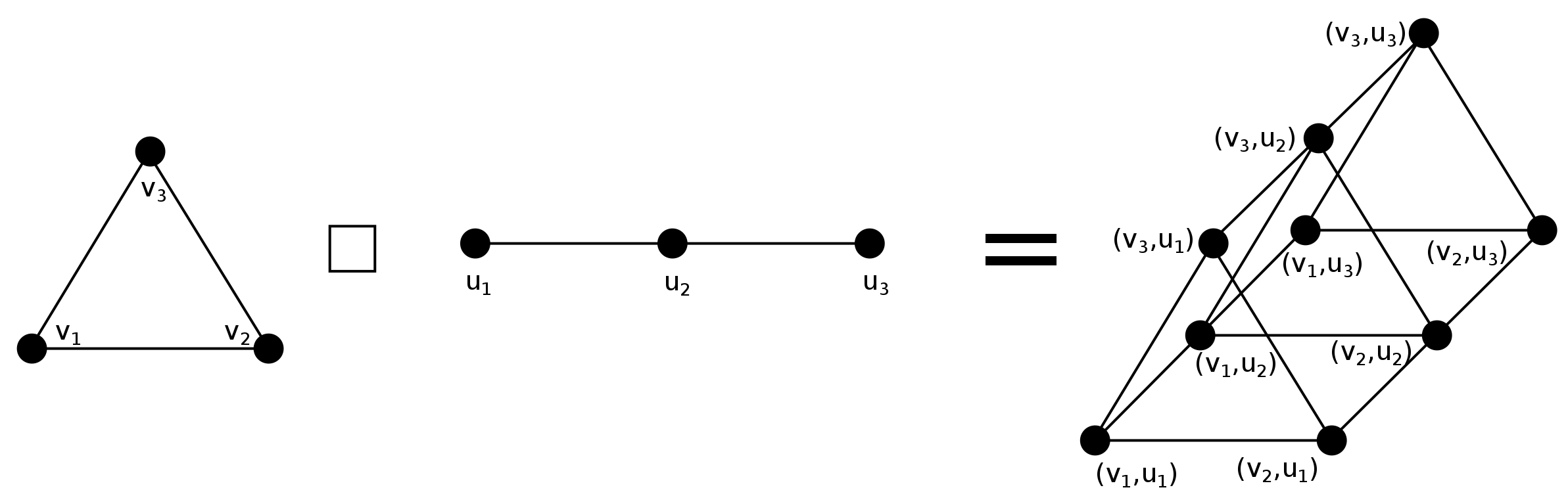}
\caption{The Cartesian Product of $K_3$ and $P_3$ with labeled vertices.}
\label{Fig:cartprod}
\end{figure}

When discussing $G=K_n$, denote vertices with integer subscripts, $V(K_n) = \{(v_i)\vert i\in \mathbb{Z}, 1\leq i\leq n\}$. Denote any vertex $v\in V(K_n^t)$ with the ordered $t$-tuple $(v_{i_1},v_{i_2},\ldots ,v_{i_t})$ where $1\leq i_j\leq n$. When indexing is necessary for elements of $V(K_n^t)$, we will use superscripts, e.g. $v^i,v^j\in V(K_n^t)$. As defined earlier, a vertex $v = (v_{i_1},v_{i_2},\ldots, v_{i_{j-1}},v_{i_j},v_{i_{j+1}},\ldots v_{i_t})$ is adjacent to every $(v_{i_1},v_{i_2},\ldots, v_{i_{j-1}},v_{i_k},v_{i_{j+1}},\ldots v_{i_t})$ where $i_j\neq i_k$. And so, if $v^1 = (v_{i_1},v_{i_2},\ldots ,v_{i_t})$ and $v^2 = (v_{j_1},v_{j_2},\ldots ,v_{j_t})$, then $d(v^1,v^2)$ is exactly the number of $k$ for which $i_k\neq j_k$. This shows that $K_n^t$ has diameter $t$.\\

This paper is primarily concerned with the more general Hamming graphs. Unless otherwise specified, the following conventions will be used throughout: $G = K_{n_1}^{t_1}\square\cdots\square K_{n_m}^{t_m}$ where $n_1 < n_2 < \cdots < n_m$. We let $\overline{t}_k:=\sum_{i=1}^kt_i$, and denote vertices $v\in V(G)$ by ordered $t$-tuples $v=(v_{i_1},\ldots,v_{i_t})$ where if $j$ is such that $\overline{t}_{k-1} < j \leq \overline{t}_k$, then $v_{i_j}\in V(K_{n_k})$. That is to say, the first $t_1$ coordinates are from $V(K_{n_1})$ and the next $t_2$ are from $V(K_{n_2})$ and so on. And just as in the case above, we have that the distance between two vertices is the number of coordinates in which the vertices differ. This also means that $G$ has diameter $\overline{t}_m$, which we will simply call $t$. Lastly, we let $N:=\vert V(G)\vert = \prod_{i=1}^mn_i^{t_i}$.\\

Niedzialomski has shown, by construction, that there exist optimal labelings (see Definition \ref{consecutivelabeling}) for $K_n^t$ where $t\leq n$ (for $n\geq 3$) and there cannot exist such labelings when $t\geq 1+\frac{n(n^2-1)}{6}$. One goal of this paper is to work towards filling the gap where $n<t<1+\frac{n(n^2-1)}{6}$ depicted in Figure \ref{Fig:gap} taken from Niedzialomski's paper.
\begin{figure}[h]
\centering
\includegraphics[scale=.75]{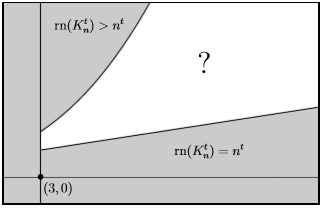}
\caption{Known results where $rn(G)$ is the smallest codomain, $\mathbb{Z}_n$, needed to label $G$.}
\label{Fig:gap}
\end{figure}

Radio labeling comes from the Channel Assignment Problem of assigning frequencies to radio transmitters, where transmitters that are closer together, must have a bigger difference in frequency to avoid interference. This problem was introduced into graph theory by Hale in 1980 \cite{Hale}.
\begin{definition}
A $k$-\defn{radio labeling} of a simple connected graph $G=(V,E)$, is a function $f:V\rightarrow\mathbb{Z}^+$ subject to the constraint
$$\vert f(v) - f(u)\vert\geq k+1-d(v,u)$$
where $v,u\in V$ are distinct and $d(v,u)$ is the distance between $v$ and $u$ in $G$. This inequality is the \defn{radio condition}.
\end{definition}

It is easy to see $k$-radio labeling as a generalization of some more familiar labelings. $1$-radio labeling is equivalent to vertex coloring since the radio condition for $k=1$ simply prohibits neighbors from having the same label. $2$-radio labeling is equivalent to $L(2,1)$-labeling which has also been studied quite heavily and was introduced in \cite{GriggsYeh}. For a survey of $2$-labeling, see \cite{Yeh}.

\begin{definition}
Of particular interest is when $k=$ diam($G$) in which case we simply call $f$ a \textbf{radio labeling}; in this case, $f$ is necessarily injective.
\end{definition}
Radio labeling was originally introduced in \cite{Chartrand}. This paper will only be concerned with consecutive radio labeling.
\begin{definition}\label{consecutivelabeling} If a radio labeling $f$ is a bijection between $V$ and $\{1,2,\ldots, \vert V\vert\}$, we call $f$ a \defn{consecutive radio labeling}. We call any $G$ for which such a labeling exists \defn{radio graceful}.
\end{definition}

Ideally there would be an algorithm for quickly computing whether a general graph is radio graceful, but no efficient algorithm currently exists. Although the exact general computational complexity of this problem is unknown, the computation is too large for general graphs of even moderate size by any known algorithm. This paper will develop a method to more efficiently compute consecutive radio labelings of Hamming graphs, making use of Niedzialomski's radio labelings induced by vertex orderings.
\begin{definition}\label{orddef}
Given a simple connected graph, $G=(V,E)$, an \defn{ordering of $V$} is an ordered list, $O = (v^1,v^2,\ldots, v^{\vert V\vert})$ such that $v^i\neq v^j$ for $i\neq j$.
\end{definition}
Given such an ordering, $O$, one can generate a radio labeling of $G$ by mapping $f(v^1)=1$, and then mapping each vertex in order so that each vertex is sent to the smallest integer possible that still satisfies the radio condition. Put more formally,
\begin{definition}\label{induceddef}
The \defn{radio labeling induced by $O$} is a function, $f:V\rightarrow\mathbb{Z}^+$, such that
$$f(v^1)=1$$
$$f(v^i) = \min\{x\in\mathbb{Z}>f(v^{i-1})\mid (1\leq j<i)\Rightarrow(\vert x-f(v^j)\vert\geq \text{diam}(G) + 1 - d(v^i,v^j))\}$$
\end{definition}

\section{Orderings of Hamming Graphs}
Our goal is to generate orderings for Hamming graphs that induce consecutive radio labelings, or determine when this is not possible. In this section, we will restate the problem of finding consecutive radio labelings as an equivalent problem involving orderings. This will be accomplished via the following,

\begin{prop}\label{rcmatham}
If $G = K_{n_1}^{t_1}\square\cdots\square K_{n_m}^{t_m}$ and $O$ is a list containing $N$ elements from $V(G)$, then $O$ is an ordering of $V(G)$ that induces a consecutive radio labeling if and only if for all $1 < i\leq N$ and all $k<t$, $v^i$ and $v^{i-k}$ share at most $k-1$ coordinates, and $O$ contains no repetition, i.e. $v^i\neq v^j$ for $i\neq j$.
\end{prop}
\begin{proof}
Suppose $O = (v^1,\ldots,v^N)$. It is true by Definition \ref{orddef} that $O$ is an ordering of $V(G)$ if and only if it contains no repetition.\\

It is also quite directly from definition that for any given $i$ and $k$ as in the premise, if $O$ induces a consecutive radio labeling, then $f(v^i) = i$ and $f(v^{i-k}) = i-k$ so that
$$f(v^i) - f(v^{i-k}) = k \geq t + 1 - d(v^i,v^{i-k})\Longleftrightarrow d(v^i,v^{i-k})\geq t- (k-1)$$
which is equivalent to the statement that $v^i$ and $v^{i-k}$ share at most $k-1$ coordinates.\\

Conversely, suppose that for all $i$ and all $k < t$, $d(v^i,v^{i-k})\geq t-(k-1)$. We need to show that the labeling induced by $O$ will be $f(v^i) = i$. It is sufficient to show that such a function $f$ is a valid radio labeling, since then it must be the induced labeling from $O$ due to Definition \ref{induceddef}. It is certainly true that any two vertices listed in $O$ less than $t$ apart will satisfy the radio condition by the previous argument (which is reversible). As for the case when $k\geq t$, then $f(v^i) - f(v^{i-k}) = k \geq t \geq t + 1 - d(v^i,v^{i-k})$  where $i$ is fixed. Therefore, $f(v^i) = i$ satisfies the radio condition as desired.
\end{proof}
\begin{corollary}\label{rcmat}
If $O$ is a list containing $n^t$ elements from $V(K_n^t)$, then $O$ is an ordering of $V(K_n^t)$ that induces a consecutive radio labeling if and only if for all $1 < i\leq n^t$ and all $k<t$, $v^i$ and $v^{i-k}$ share at most $k-1$ coordinates, and $O$ contains no repetition, i.e. $v^i\neq v^j$ for $i\neq j$.
\end{corollary}

When $G = K_{n_1}^{t_1}\square\cdots\square K_{n_m}^{t_m}$, we will write any ordering, $O = (v^1,\ldots,v^N)$, as a $N\times t$ matrix where the $i$\textsuperscript{th} row of $O$ is $v^i = (v^i_1,v^i_2,\ldots ,v^i_t)$ where each $v^i_j = v_l\in V(K_{n_k})$ for some $l$ and the appropriate $n_k$. In particular, orderings of $K_n^t$ are $n^t\times t$ matrices containing elements from $V(K_n)$ (See Figure \ref{Fig:K32Ord}). From now on, all orderings will be these matrices.\\

Proposition \ref{rcmatham} allows us to essentially rewrite the radio condition in this new context of $N\times t$ matrices. This transforms our problem from graph labeling to trying to generate a matrix satisfying certain properties; namely, that row $i$ may be identical to row $i-k$ in at most $k-1$ places.\\

However, this transition from graph labeling to matrix generation means that we not only need to mind the radio condition but also avoid repetition as is mentioned in Proposition \ref{rcmatham}. Since our focus will almost always be on the radio condition with no regard for repetition, the following definition is necessary.
\begin{definition}
A \textbf{weak ordering} of $V(K_{n_1}^{t_1}\square\cdots\square K_{n_m}^{t_m})$ is an $N\times t$ matrix where the elements of column $j$ are chosen from $V(K_{n_k})$ when $\overline{t}_{k-1}<j\leq\overline{t}_k$. The set of all orderings is the subset of the set of all weak orderings whose rows have no repetition.
\end{definition}

\begin{figure}[h]
\centering
\includegraphics[scale=.15]{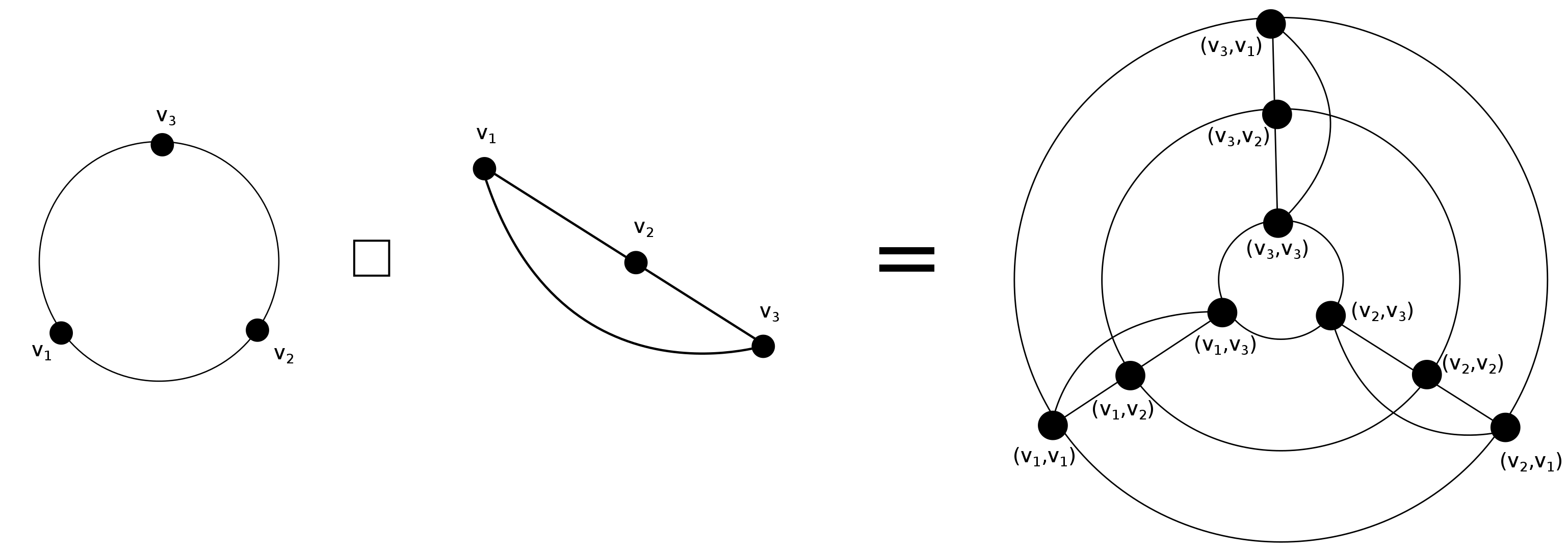}
\includegraphics{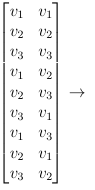}\includegraphics[scale=.15]{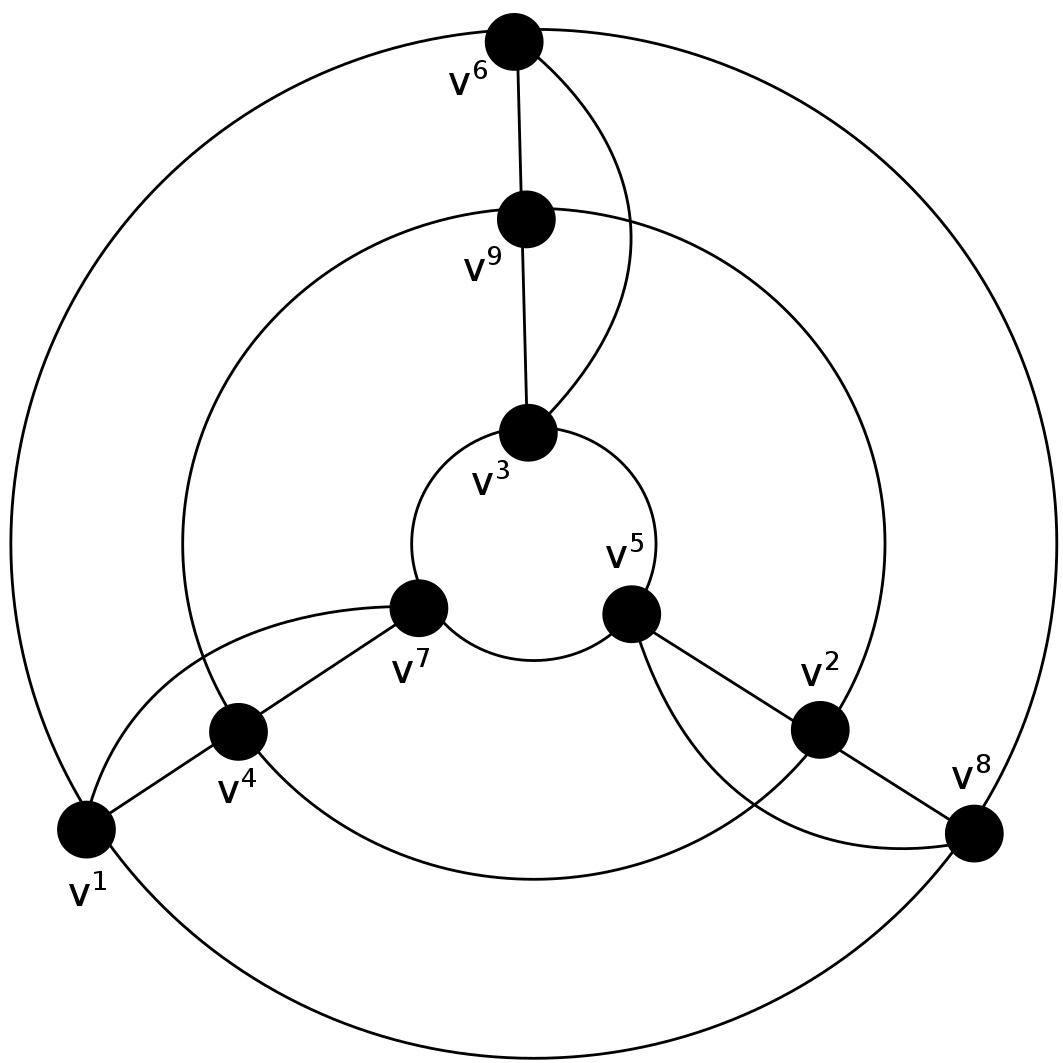}
\caption{The Matrix seen above (bottom) is an ordering of the vertices of $K_3^2$ (top) that induces a consecutive radio labeling.}
\label{Fig:K32Ord}
\end{figure}

Lastly, the following definition and lemma will be helpful anytime we wish to relabel our $v_i$ in orderings.
\begin{definition}\label{permutecoldef}
Define $O_{u,\sigma}$ as follows. Given $O$, an ordering of $V(G)$, let $u = (v_j^i)_{i=1}^{N}$ be a column of $O$, and let $\sigma\in S_{n_k}$ be a permutation where $k$ is such that $u$ has elements from $V(K_{n_k})$. Then replacing $u$ in $O$ with the new column $(v_{\sigma(l_{i,j})})_{i=1}^{N}$ (where $v_j^i = v_{l_{i,j}}$) yields a new ordering, call it $O_{u,\sigma}$.
\end{definition}
\begin{lemma}\label{permutecol}
If $O$ induces a consecutive radio labeling (of $V(K_{n_1}^{t_1}\square\cdots\square K_{n_m}^{t_m})$), $u$ is a column of $O$, and $\sigma\in S_{n_k}$ for the appropriate $k$, then $O_{u,\sigma}$ also induces a consecutive radio labeling.
\end{lemma}
\begin{proof}
This is a direct result of Proposition \ref{rcmatham}, since permuting the elements of a column as described does not change the number of coordinates shared between any pair of rows.
\end{proof}
One way to view the Cartesian product in this context is to see each coordinate of a vertex, and hence column of an ordering $O$, as corresponding to a copy of $K_{n_l}$ for some $l$, where each vertex (row) in the product is a choice of one vertex from each of these complete graphs of which there are $t$ in total. In light of this, the process of permuting a column as in Definition \ref{permutecoldef} is exactly the same as relabeling the vertices of the copy of $K_{n_k}$ to which this column corresponds. In general, this fact that each column is somewhat independent (as far as Proposition \ref{rcmatham} is concerned) is what allows us to generalize many results from graphs of the form $K_n^t$ to more general Hamming graphs.

\section{Bounds on labeling}
In Niedzialomski's paper \cite{Amanda}, it is shown that for any graph $G$, there is some integer $t$ for which $G^t$ is not radio graceful. Specifically there is a bound for $K_n^t$ stated here as Corollary \ref{upperbound}. I will present an analogous bound for the more general case of Hamming Graphs. In the following proofs of bounds, we will consider orderings, and derive a contradiction if $t$ is too much larger than $n$.\\

I will demonstrate the idea behind the proof method of the bound with the example of $K_3^5$ which I will show to be not radio graceful. (It may be helpful to follow along in Figure \ref{Fig:K35Ord})
\begin{example}
Suppose for the sake of contradiction that there exists an ordering, $O$, of $V(K_3^5)$ which induces a consecutive radio labeling. Consider the segment of $O$ from $v^i$ to $v^{i+3}$ for any appropriate $i$ (i.e. $1\leq i\leq 3^5-3$). Due to Lemma \ref{permutecol}, we can permute each column so that $v_i = (v_1,v_1,v_1,v_1,v_1)$ and $v^{i+1} = (v_2,v_2,v_2,v_2,v_2)$, we can do this since we know that $v^i$ and $v^{i+1}$ share no coordinates. Next, $v^{i+2}$ can share at most one coordinate with $v^i$ and must be different from both $v^i$ and $v^{i+1}$ in all other coordinates. Thus, using Lemma \ref{permutecol} again, we may now permute the last four columns of $O$ so that the last four coordinates of $v^{i+2}$ become $v_3$ (different from both $v_1$ and $v_2$). Lastly, consider $v^{i+3}$; the first coordinate is no longer under our control to permute as it was not fixed in $v^{i+2}$ meaning that we can't ensure that any permutation involving this column wouldn't change our current setup. However, we do know that $v^{i+3}$ may share at most two places with $v^i$, and one more with $v^{i+2}$. However, this leaves still one coordinate which must then be different from all the previous vertices in this segment, and this is impossible since our only choices are $v_1,v_2,$ and $v_3$. Therefore, there does not exist an ordering, $O$, that induces a consecutive radio labeling on $K_3^5$. Additionally, from this construction we can conclude that in any valid ordering for $K_3^4$, $v^i$ must share exactly one coordinate with $v^{i+2}$, exactly two with $v^{i+3}$, and $v^{i+1}$ must also share a coordinate with $v^{i+3}$, otherwise we have the same problem.
\begin{figure}[h]
\begin{center}
$\begin{bmatrix}
v^i\\v^{i+1}\\v^{i+2}\\v^{i+3}
\end{bmatrix} = \left[\begin{array}{ccccccccccc}
v_1&v_1&v_1&v_1&v_1\\
v_2&v_2&v_2&v_2&v_2\\
v_1?&v_3&v_3&v_3&v_3\\
?&v_2?&v_1?&v_1?&\textbf{X}
\end{array}\right]$
\end{center}
\caption{A segment of an ordering of $K_3^5$ showing that there is no possible fourth vertex.}
\label{Fig:K35Ord}
\end{figure}
\end{example}

See Figure \ref{Fig:K411Ord} for a demonstration of this contradiction with $K_4^{11}$ (because $1+\frac{4(16-1)}{6}=11$), although this figure can also be viewed as a demonstration for $K_3^4\square K_4^7$.

\begin{figure}[h]
\begin{center}
$\begin{bmatrix}
v^i\\v^{i+1}\\v^{i+2}\\v^{i+3}\\v^{i+4}
\end{bmatrix} = \left[\begin{array}{ccccccccccc}
v_1&v_1&v_1&v_1&v_1&v_1&v_1&v_1&v_1&v_1&v_1\\
v_2&v_2&v_2&v_2&v_2&v_2&v_2&v_2&v_2&v_2&v_2\\
v_1?&v_3&v_3&v_3&v_3&v_3&v_3&v_3&v_3&v_3&v_3\\
?&v_2?&v_1?&v_1?&v_4&v_4&v_4&v_4&v_4&v_4&v_4\\
?&?&?&?&v_3?&v_2?&v_2?&v_1?&v_1?&v_1?&\textbf{X}
\end{array}\right]$
\end{center}
\caption{A segment of an ordering of $K_4^{11}$ showing that there is no possible fifth vertex.}
\label{Fig:K411Ord}
\end{figure}

The big idea is that because of the radio condition (viewed as in Proposition \ref{rcmatham}), we only have so many columns in which we can differ with previous vertices, so that many coordinates must be different from those above them and eventually this isn't possible because we run out of options (i.e. elements of $V(K_n)$).\\

\begin{definition}\label{alphadef} Given $O$, an ordering of $V(G)$ that induces a consecutive radio labeling, define $\alpha_k^i$ to be the number of columns, $j$, of $O$ for which the collection $\{v_j^i,v_j^{i+1},\ldots,v_j^{i+k}\}$ is made up of pairwise distinct elements. For example, for any $i$, $\alpha_0^i = \alpha_1^i = t$ and $\alpha_2^i$ is either $t$ or $t-1$.
\end{definition}
\begin{prop}\label{alphastep}
For any non-negative integer $k$ and any possible $i$,
$$\alpha_{k+1}^i\geq \alpha_k^i - \sum_{b=1}^kb.$$
\end{prop}
\begin{proof}
This is a direct consequence of Proposition \ref{rcmatham} since $v^{i+k+1}$ may share at most $k+(k-1)+\cdots+1+0$ total coordinates with $v^i,v^{i+1},\ldots,v^{i+k}$ so that there are at most $\sum_{b=1}^kb$ columns not counted in $\alpha_{k+1}^i$ that were counted in $\alpha_k^i$.
\end{proof}
\begin{prop}\label{alphabound}
For any $1\leq k\leq m$ and any possible $i$,
$$\alpha^i_{n_k}\leq t - \overline{t}_k$$
\end{prop}
\begin{proof}
For any such $k$ and $i$, there are $\overline{t}_k$ columns of $O$ which are chosen from $V(K_{n_i})$ where $n_i\leq n_k$ so that none of these columns can count towards $\alpha_{n_k}^i$ since it cannot be that $\{v^i_j,\ldots,v^{i+n_k}_j\}$ are pairwise distinct if these $n_k+1$ elements are chosen from a set of $n_k$ or fewer elements. There are $t$ total columns so that $t-\overline{t}_k$ is an upper bound on $\alpha_{n_k}^i$. 
\end{proof}

With all of this under our belt, we are now ready to prove the bound. Note that the following corollary is found in (\cite{Amanda}, Corollary 12) and this is simply a different and more direct proof (in that if $G = K_n^t$, i.e. $m=1$, then the following is a direct proof) using the construction above that is defined for the more general Hamming graphs.
\begin{theorem}\label{upperboundgeneral}
Given $G = K_{n_1}^{t_1}\square\cdots\square K_{n_m}^{t_m}$, if $\overline{t}_k\geq 1 + \frac{n_k(n_k^2-1)}{6}$ for any $k$, then $G$ is not radio graceful.
\end{theorem}
\begin{proof}
Suppose, for the sake of contradiction, that $G$ is radio graceful where $\overline{t}_k\geq 1 + \frac{n_k(n_k^2-1)}{6}$. Then there exists an ordering $O$ of $V(G)$ that induces a consecutive radio labeling. Let $i$ be any fixed positive integer less than $N-n_k$. By Proposition \ref{alphastep}, the following string of inequalities hold,
$$\alpha_{n_k}^i\geq \alpha_{n_k-1}^i-\sum_{b=1}^{n_k-1}b\geq \alpha_{n_k-2}^i - \sum_{b=1}^{n_k-2}b - \sum_{b=1}^{n_k-1}b\geq\cdots\geq\alpha_1^i - \sum_{a=1}^{n_k-1}\sum_{b=1}^ab = t - \frac{n_k(n_k^2-1)}{6} > t-\overline{t}_k$$
But $\alpha_{n_k}^i > t-\overline{t}_k$ contradicts Proposition \ref{alphabound}. Therefore $G$ is not radio graceful when $\overline{t}_k\geq 1 + \frac{n_k(n_k^2-1)}{6}$.
\end{proof}
\begin{corollary}\label{upperbound}
If $t\geq 1 + \frac{n(n^2-1)}{6}$, then $K_n^t$ is not radio graceful.
\end{corollary}

Furthermore, we can by a similar method, using Definition \ref{alphadef}, prove something constructive about graphs that are right underneath our bound.

\begin{theorem}\label{BoundaryForcedGeneral}
Let $G = K_{n_1}^{t_1}\square\cdots\square K_{n_m}^{t_m}$ with $\overline{t}_k = \frac{n_k(n_k^2-1)}{6}$ for some $k$. If $O$ is any ordering of $V(G)$ that induces a consecutive radio labeling, then for every $i$ and every $j\leq n_k$, $v^i$ and $v^{i+j}$ must share exactly $j-1$ coordinates.
\end{theorem}
\begin{proof}
Suppose, for the sake of contradiction, that for some $i$ and some $j\leq n_k$, $v^i$ and $v^{i+j}$ don't share exactly $j-1$ coordinates. Then it must be the case that they share fewer than $j-1$ coordinates due to Proposition \ref{rcmatham}. This implies that $$\alpha_j^i>\alpha_{j-1}^i - \sum_{b=1}^{j-1}b$$ by Proposition \ref{alphastep} where we cannot have equality since $v^i$ and $v^{i+j}$ share fewer than $j-1$ coordinates. Thus, we can again use Proposition \ref{alphastep} to get the following string of inequalities,
$$\alpha_{n_k}^i\geq\alpha_{n_k-1}^i - \sum_{b=1}^{n_k-1}b\geq\cdots\geq\alpha_j^i-\sum_{a=j}^{n_k-1}\sum_{b=1}^ab > \alpha_{j-1}^i-\sum_{a=j-1}^{n_k-1}\sum_{b=1}^ab\geq\cdots\geq \alpha_1^i - \sum_{a=1}^{n_k-1}\sum_{b=1}^ab = t - \overline{t}_k$$
But $\alpha_{n_k}^i > t-\overline{t}_k$ contradicts Proposition \ref{alphabound}. Therefore, if $O$ induces a consecutive radio labeling, then $v^i$ and $v^{i+j}$ must share exactly $j-1$ coordinates.
\end{proof}
\begin{corollary}\label{BoundaryForced}
In any ordering, $O$, of the vertices of $K_n^{\frac{n(n^2-1)}{6}}$ that induces a consecutive radio labeling, $v^i$ and $v^{i+j}$ for every $i$ and every $j\leq n$ must share exactly $j-1$ coordinates.
\end{corollary}

Lastly, another result from Niedzialomski's paper that I will not be showing here is that if $t\leq n$, then $G = K_n^t$ is radio graceful. She cleverly constructs a consecutive radio labeling using permutations of smaller segments of an ordering for $G$, but this construction does not give a consecutive radio labeling for $K_n^{n+1}$ so we must find a different approach to generating labelings in order to fill the gap in Figure \ref{Fig:gap}.

\section{Generating Orderings}
Orderings have allowed us to restate the problem of radio labeling as one of constructing matrices that satisfy certain conditions. This abstraction not only makes the problem easier to reason about, especially with Hamming Graphs, but also leads us to certain results rather directly. In this section, I will introduce a new family of matrices that fully encapsulate radio labeling of Hamming Graphs in the hopes that this new paradigm will be helpful in finding new results and closing the gap in Figure \ref{Fig:gap}. These new matrices can be viewed as containing `instructions' for generating an ordering; however, we will often be able to restate the radio condition in the context of these generating matrices so that we need not consider orderings anymore, just as orderings allowed us to essentially disregard the graph concept.\\

Proposition \ref{rcmatham} states the radio condition for orderings in such a way that allows us to view columns as somewhat independent, as has been mentioned earlier. In particular, if we are attempting to generate an ordering and we come to any entry, $v^i_j$, we need only inspect column $j$ when asking how this entry affects our radio condition considerations. As such, we will restrict ourselves to generating columns of orderings and find a mechanism for detecting repetition within a column, i.e. when two rows share a coordinate in this column.

\begin{definition}
Let $\Delta_n\subset S_n$ containing $n-1$ elements $\Delta_n = \{f_2,\ldots,f_n\}$ such that $f_k(k)=1$ for each $k$. Any $\Delta_n$ satisfying these conditions is called an \textbf{instruction set}. Call the collection $\{\Delta_n^i:S_n^{i-1}\rightarrow\mathcal{P}(S_n)\mid 2\leq i\leq N\}$, denoted simply by $\Delta_n^i$, an \textbf{instruction set generator} if every element in the image of each $\Delta_n^i$ is an instruction set. An instruction set generator can be viewed as a function that takes as input the current state of the column and generates an instruction set containing the possible next instructions. We will use the convention that $\sigma_1\sigma_2 = \sigma_2\circ\sigma_1$ when denoting composition of permutations.
\end{definition}

\begin{definition}\label{AnBndef}
Given integers $N,n\geq 3$, and an instruction set generator, $\Delta_n^i$, define
\begin{align*}
A_n &= \left\{\vec{u} = \begin{pmatrix}
v_{i_1}=v_1\\v_{i_2}=v_2\\v_{i_3}\\\vdots\\v_{i_N}
\end{pmatrix}\mid v_{i_j}\in V(K_n), v_{i_j}\neq v_{i_{j+1}}\right\}\\
B_n &= \left\{\vec{u'} = \begin{pmatrix}
\sigma_1 = id\\\sigma_2 = f_2\\\sigma_3\\\vdots\\\sigma_N
\end{pmatrix}\mid\sigma_i\in\Delta_n^i(\sigma_1,\ldots,\sigma_{i-1})\text{ for }i\geq 2\right\}.
\end{align*}
Elements of $A_n$ can be thought of columns of orderings which correspond to $K_n$ while elements of $B_n$ will be the columns of our new matrix that can be viewed as generating a corresponding element in $A_n$.
\end{definition}
We will construct a 1-1 correspondence between $A_n$ and $B_n$ using the following,
\begin{definition}\label{actiondef}
Let $D_n$ be the subset of $(V(K_n))^n$ whose coordinates are all distinct, i.e. $(v_{s_1},\ldots,v_{s_n})\in D_n$ if and only if $v_{s_i}\in V(K_n)$ for all $1\leq i\leq n$ and $v_{s_i}\neq v_{s_j}$ for all $1\leq i < j\leq n$. Define an action from $S_n$ on $D_n$ by $\sigma\cdot(v_{s_1},\ldots,v_{s_n}) = (v_{s_{\sigma^{-1}(1)}},\ldots,v_{s_{\sigma^{-1}(n)}})$ where $\sigma\in S_n$. We have $S_n$ act on $D_n$ in such a way so that we can talk about $\sigma = (135)$, for example, as being the instruction that sends the first coordinate of an element of $D_n$ to the third place, the third to the fifth and the fifth back to the first.
\end{definition}
\begin{definition}\label{phidef}
We now construct our bijection between $A_n$ and $B_n$. Define $(D_n^N)'$ to be the subset of $D_n^N$ whose elements $\begin{pmatrix}(v_1^1,\ldots,v_n^1)\\\vdots\\(v_1^N,\ldots,v_n^N)\end{pmatrix}$ satisfy the following constraints: that $v_1^j\neq v_1^{j+1}$ for all $j$, $v_i^1 = v_i$ for $1\leq i\leq n$ and $v_1^2 = v_2$.\\
Let $\varphi:B_n\rightarrow (D_n^N)'$ be defined as follows, if $\vec{u'} = \begin{pmatrix}
\sigma_1\\\vdots\\\sigma_N
\end{pmatrix}\in B_n$, then let
$$\varphi(\vec{u'}) = \begin{pmatrix}
o_1 = (v_1,\ldots,v_n)\\o_2 = \sigma_2\cdot o_1\\\vdots\\o_N = \sigma_N\cdot o_{N-1}
\end{pmatrix}.$$
Note that $\varphi(\vec{u'})\in(D_n^N)'$ because $\sigma_2=f_2$ so that $o_2$ has $2$ as its first coordinate, and since $\sigma^{-1}(1)\neq 1$ for all $\sigma\in\Delta_n$.\\
Next, let $\pi_1:(D_n^N)'\rightarrow A_n$ be the projection to the first coordinate, i.e. $$\pi_1\begin{pmatrix}(v_1^1,\ldots,v_n^1)\\\vdots\\(v_1^N,\ldots,v_n^N)\end{pmatrix} = \begin{pmatrix}v_1^1\\\vdots\\v_1^N\end{pmatrix}.$$
Finally, define $\phi_n:B_n\rightarrow A_n$ by composition
$$\phi_n := \pi_1\circ\varphi.$$
\end{definition}

\begin{theorem}\label{corr}
If $n,N\geq 3$ are any integers, and $\Delta_n^i$ is any instruction set generator, then $\phi_n$, as in Definition \ref{phidef}, is a 1-1 correspondence between $A_n$ and $B_n$.
\end{theorem}
\begin{proof}
Since $A_n$ and $B_n$ both have $(n-1)^{N-2}$ elements, we need only show that $\phi_n$ is surjective. Given $\vec{u} = \begin{pmatrix}
v_{i_1}\\\vdots\\v_{i_N}
\end{pmatrix}\in A_n$, let $\vec{u''} = \begin{pmatrix}
o_1\\\vdots\\o_N
\end{pmatrix}$ where $o_1=(v_1,\ldots,v_n)$, $\sigma_1=id\in S_n$, and if $o_j = (v_{s_1},\ldots,v_{s_n})$ and $i_{j+1} = s_k$, then $\sigma_{j+1} = f_k\in\Delta_n^{j+1}(\sigma_1,\ldots,\sigma_j)$ and $o_{j+1} = \sigma_{j+1}\cdot o_j$. By construction, we have that $\pi_1(\vec{u''}) = \vec{u}$ and that $\vec{u''}$ is in $Im(\varphi)$, specifically, $\vec{u'} = \begin{pmatrix}
\sigma_1\\\vdots\\\sigma_N
\end{pmatrix}\in B_n$ with $\phi_n(\vec{u'}) = \vec{u}$ making $\phi_n$ surjective and therefore bijective.
\end{proof}

With this theorem, we are ready to restate the problem of radio labeling as a problem of finding matrices that generate orderings.

\begin{definition}
Given $G = K_{n_1}\square\cdots\square K_{n_t}$, a Hamming graph (the $n_i$ are not necessarily distinct), a \textbf{weak order-generator}, $O'$, for $V(G)$ (with respect to choices of $\Delta_{n_k}^i$ for each $n_k$) is an $N\times t$ matrix where every entry of the first row is $id$, every entry of the second row is $f_2$, and for all but the first row, elements of column $j$ are chosen using $\Delta_{n_j}^i$. We denote the $j$th element of the $i$th row by $\sigma_i^j$.
\end{definition}
\begin{definition}
Given $G = K_{n_1}\square\cdots\square K_{n_t}$ and choices of $\Delta_{n_k}^i$ for each $n_k$, define $\Phi_G$ to be a function from the set of weak order-generators for $V(G)$ to the set of all weak orderings of $V(G)$ whose first two rows are $(v_1,\ldots,v_1)$ and $(v_2,\ldots,v_2)$. Given $O'=[c_1\ \cdots\ c_t]$, $\Phi_G(O') := [\phi_{n_1}(c_1)\ \cdots\ \phi_{n_t}(c_t)]$, i.e. $\Phi_G$ applies the corresponding $\phi_n$ (using the corresponding $\Delta_{n_j}^i$) to each column. Since each of the $\phi_n$ are bijections, so is $\Phi_G$.
\end{definition}
Notice that given any ordering that induces a consecutive radio labeling of $G$, we can always use Lemma \ref{permutecol} to get a new ordering whose first two rows are $(v_1,\ldots,v_1)$ and $(v_2,\ldots,v_2)$, as required above, that also induces a consecutive radio labeling.
\begin{definition}
An \textbf{order-generator} of $V(G)$ (with respect to choices of $\Delta_{n_k}^i$ for each $n_k$) is a weak order-generator, $O'$, such that $\Phi_G(O')$ is an ordering, i.e. $\Phi_G(O')$ has no row repetition. Note that if $\Phi_G$ is restricted to only order-generators, it becomes a bijection between all order-generators and all orderings.
\end{definition}

The following is an example where an ordering of $K_3^2$ that induces a consecutive radio labeling is shown next to the its generator and intermediate construction where $$\Delta_3 = \{f_2 = (12),f_3 = (123)\}.$$
is used for both columns. That is, the constant function $\Delta_3^i = \Delta_3$.
$$O=\begin{bmatrix}
v_1&v_1\\
v_2&v_2\\
v_3&v_3\\
v_1&v_2\\
v_2&v_3\\
v_3&v_1\\
v_1&v_3\\
v_2&v_1\\
v_3&v_2
\end{bmatrix}\leftrightarrow
\varphi(O') = 
\begin{pmatrix}
(v_1,v_2,v_3)&(v_1,v_2,v_3)\\
(v_2,v_1,v_3)&(v_2,v_1,v_3)\\
(v_3,v_2,v_1)&(v_3,v_2,v_1)\\
(v_1,v_3,v_2)&(v_2,v_3,v_1)\\
(v_2,v_1,v_3)&(v_3,v_2,v_1)\\
(v_3,v_2,v_1)&(v_1,v_3,v_2)\\
(v_1,v_3,v_2)&(v_3,v_1,v_2)\\
(v_2,v_1,v_3)&(v_1,v_3,v_2)\\
(v_3,v_2,v_1)&(v_2,v_1,v_3)
\end{pmatrix}
\leftrightarrow
O'=\begin{bmatrix}
id&id\\
f_2&f_2\\
f_3&f_3\\
f_3&f_2\\
f_3&f_2\\
f_3&f_3\\
f_3&f_2\\
f_3&f_2\\
f_3&f_3
\end{bmatrix}$$
The reader is encouraged to write out their own example, even if it be just for a single column, and use this as reference later in this section and the next.\\

Now that we have the construction of order-generators, we wish to find a way of detecting repetition in the orderings that are generated.

\begin{lemma}\label{replem}
Let $A_n$ and $B_n$ be as usual (for a fixed $\Delta_n^i$). If $\phi_n(\begin{pmatrix}\sigma_1\\\vdots\\\sigma_N\end{pmatrix}) = \begin{pmatrix}v_{i_1}\\\vdots\\v_{i_{N}}\end{pmatrix}$, then
$$v_{i_j} = v_{i_{j+k}} \Leftrightarrow \sigma_{j+1}\sigma_{j+2}\cdots\sigma_{j+k}(1) = 1.$$
\end{lemma}
\begin{proof}
Let $\vec{u'} = \begin{pmatrix}\sigma_1\\\vdots\\\sigma_N\end{pmatrix}$ and $\vec{u} = \begin{pmatrix}v_{i_1}\\\vdots\\v_{i_{N}}\end{pmatrix}$ above. Define $\vec{u''} = \varphi(\vec{u'}) = \begin{pmatrix}o_1\\\vdots\\o_N\end{pmatrix}$. By Definition \ref{phidef}, we have that
$$o_{j+k} = \sigma_{j+k}\cdot o_{j+k-1} = (\sigma_{j+k-1}\sigma_{j+k})\cdot o_{j+k-2} = \cdots = (\sigma_{j+1}\cdots\sigma_{j+k})\cdot o_j,$$ and that $\pi_1(\vec{u''}) = \vec{u}$ meaning that the first coordinate of $o_j$ is $v_{i_j}$ and the first coordinate of $o_{j+k}$ is $v_{i_{j+k}}$, so that we can conclude $$v_{i_j} = v_{i_{j+k}}\Leftrightarrow \pi^1(o_j) = \pi^1((\sigma_{j+1}\cdots\sigma_{j+k})\cdot o_j)\Leftrightarrow \sigma_{j+1}\cdots\sigma_{j+k}(1) = 1$$
where $\pi^1(v_{s_1},\ldots,v_{s_n}) := v_{s_1}$.
\end{proof}

In order to rewrite the conditions for a consecutive radio labeling more explicitly for $O'$, we wish to keep track of when a run of instructions in a column of $O'$ maps $1$ to itself, as this corresponds with repetition in that column of $O$ by the previous lemma.

\begin{definition} Given an instruction set generator $\Delta_n^i$, define $\Lambda_s^i$ to be the set of all runs of $s$ instructions, starting at position $i$, that map $1$ to itself. That is,
$$\Lambda_s^i(\rho_1,\ldots,\rho_{i-1}) := \{\sigma_1\sigma_2\cdots\sigma_s\mid\sigma_k\in\Delta_n^{i+k-1}(\rho_1,\ldots,\rho_{i-1},\sigma_1,\ldots,\sigma_{k-1}),\sigma_1\cdots\sigma_s(1)=1\}.$$
In the presence of multiple $\Delta_{n_k}^i$, we will denote the $\Lambda_s$ corresponding to $\Delta_{n_k}^i$ by $\Lambda_{s,k}^i$.
\end{definition}

Hence, if one can characterize $\Lambda_s^i$ for a given $\Delta_n^i$, then one can restate the conditions for a consecutive radio labeling easily by combining Lemma \ref{replem} with Proposition \ref{rcmatham},

\begin{theorem}\label{LambdasGeneral}
$O'$ generates an ordering that induces a consecutive radio labeling of $G = K_{n_1}\square\cdots\square K_{n_t}$ if and only if there is no repetition in $O:=\Phi_G(O')$, and for all $i<N$ and for every $s<t$, at most $s-1$ values of $j$ result in runs, $\sigma_{i-s+1}^j\cdots \sigma_i^j$, contained in $\Lambda_{s,j}^{i-s+1}(\sigma_1^j,\ldots,\sigma_{i-s}^j)$.
\end{theorem}
\begin{corollary}\label{Lambdas}
$O'$ generates an ordering that induces a consecutive radio labeling of $K_n^t$ if and only if there is no repetition in $O:=\Phi_G(O')$, and for all $i<n^t$ and for every $s<t$, at most $s-1$ values of $j$ result in runs, $\sigma_{i-s+1}^j\cdots \sigma_i^j$, contained in $\Lambda_s^{i-s+1}(\sigma_1^j,\ldots,\sigma_{i-s}^j)$.
\end{corollary}

\section{Examples and an Ordering of $K_3^4$}
We will now consider some simple examples of choices for $\Delta_n^i$ with characterizations of their corresponding $\Lambda_n^i$, and we will conclude with a valid ordering of $K_3^4$ which was previously not known to be radio graceful.\\
For the remainder of this section, if there is ever a place where any element of $\Delta_n$ can be used in a run, then that we will simply use $f$ to denote this. Also, for any $k$, let $\overline{f_k}$ denote any instruction other than $f_k$, i.e. the complement of $\{f_k\}$ in $\Delta_n$ (the relevant instruction set), and let $\overline{f_k^s}$ denote a run of $s$ elements in the complement of $\{f_k\}$ in $\Delta_n$. Lastly, let $f_{<k}$ denote any element in $\{f_2,\ldots,f_{k-1}\}$ and $f_{>k}$ denote any element in $\{f_{k+1},\ldots,f_n\}$.

\begin{example}
The simplest instruction set generator is the constant function
$$\Delta_n^i = \Delta_n = \{f_k = (1k)\mid 2\leq k\leq n\}.$$
In this case, we can determine $\Lambda_s^i$ recursively as the constant function
$$\Lambda_s^i = \{xf_k\overline{f_k^l}f_k\mid 2\leq k\leq n,l\geq 0,x\in\Lambda_{s-l-2}^i\}.$$
\end{example}

\begin{example}\label{LRU}
Another constant function example is
$$\Delta_n^i = \Delta_n = \{f_k = (12\cdots k)\mid 2\leq k\leq n\}$$
which I call the Least Recently Used (LRU) instruction set because each element of $\varphi(O')$ orders $V(K_n)$ by least recently used first (where they are 'used' in that column of $\Phi_n(O')$). In this case $\Lambda_s^i$ is a bit more complicated, but can once again be characterized recursively as the constant function
$$\Lambda_s^i = \{xff_{>2}f_{<3}^{n_3}f_{>3}f_{<4}^{n_4}f_{>4}\cdots f_{<k}^{n_k}f_k\mid 2\leq k\leq n, x\in\Lambda_{s-(k-1)-\sum_{i=3}^kn_i}^i\}.$$
\end{example}

\begin{example}\label{LTU}
Yet another example of a constant function example is
$$\Delta_n^i = \Delta_n = \{f_2 = (12),f_k = (12k)\mid 3\leq k\leq n\}.$$
This instruction set guarantees that the first two coordinates of each element of $\varphi(O')$ are the most recent two coordinates (which are guaranteed to be different by the constraint from the radio condition on orderings, as seen in $A_n$), and is the simplest instruction set that does so.\\
As usual, we can characterize $\Lambda_s^i$ recursively as the constant function
$$\Lambda_s^i = \{xff_2\vert x\in\Lambda_{s-2}^i\}\cup\{xff_k\overline{f_k^{s'}}f_k\vert 3\leq k\leq n,x\in\Lambda_{s-s'-3}^i\}.$$
\end{example}

\subsection{Order-Generators as a Generalization of Orderings}
This last example will demonstrate that the framework of instruction set generators and their associated order-generators is more general than using orderings. Consider the instruction set generator that only considers the most recent element,
$$\Delta_n^i(\sigma_1,\ldots,\sigma_{i-1}) = \Delta_n^i(\sigma_{i-1}) = \{f_{k'} = (1k'),f_k = (1kk')\mid 2\leq k\leq n, k\neq k'\}$$
where $k'$ is such that $\sigma_{i-1} = f_{k'}\in\Delta_n^{i-1}(\sigma_{i-2})$. If we use the naming convention of renaming $id = f_1$, and $f_k' = f_1$, doing this every other time if there are consecutive instructions with the same subscript, then we get that $\phi_n(\begin{pmatrix}
v_{i_1}\\\vdots\\v_{i_N}
\end{pmatrix}) = \begin{pmatrix}
f_{i_1}\\\vdots\\f_{i_N}
\end{pmatrix}$. In particular,
$$\Lambda_s^i(\sigma_1,\ldots,\sigma_{i-1}) = \Lambda_s^i(\sigma_{i-1}) = \{f^{s-1}f_k\mid k\text{ is such that }\sigma_{i-1} = f_k\}.$$
Therefore, if we look at order generators as just another matrix with integer subscripts that must satisfy some conditions, then this example shows that we can recover our original matrices (orderings) and their constraints in this new framework.

\subsection{Labeling $K_3^4$}
Fix $\Delta_3^i$ for all columns to be the constant function $\Delta_3^i=\Delta_3 = \{f_2 = (12), f_3 = (123)\}$ (this can be thought of as using $\Delta_n$ from Example \ref{LRU} or from Example \ref{LTU}). As we are only concerned with $\Lambda_s$ for $s<4$ here, we need only consider $\Lambda_2 = \{f_2^2,f_3f_2\}$ and $\Lambda_3 = \{f_2f_3^2,f_3^3\}$. Thus, we can have at most one $f_2$ instruction in every row (after the second), and at most two $f_3$ instructions in the same column as an $f_3$ in the row above. This is because any $f_2$ will necessarily cause either an $f_2^2$ or $f_3f_2$, and any two consecutive $f_3$ instructions will necessarily cause either an $f_2f_3^2$ or an $f_3^3$. Since we cannot have $3$ columns with consecutive $f_3$ instructions, we must have at least (which is equivalent to having exactly) one $f_2$ instruction in every row, and that $f_2$ must be in a different coordinate than the row above it. This is because having four $f_3$ in any row would necessarily create a problem in both the next and previous rows, and we cannot place two consecutive $f_2$ in the same column as this would also result in three runs in $\Lambda_3$. (Note that the result of Corollary \ref{BoundaryForced} is seen to come out here).\\

In this manner, as there are now only four states for every row of $O'$ (corresponding to the position of $f_2$), the problem of generating a consecutive radio labeling for $K_3^4$ has been reduced to choosing a coordinate in which to place $f_2$ in each row, with the restriction that the same coordinate cannot be chosen twice in a row, and that there can be no repetition.\\

It was previously unknown whether $K_3^4$ is radio graceful, but using the reduction presented, a backtrack searching algorithm (that was used to avoid repetition) found many orderings that induce consecutive radio labelings. Here is one of them:
$$O=\begin{bmatrix}
v_1&v_1&v_1&v_1\\
v_2&v_2&v_2&v_2\\
v_1&v_3&v_3&v_3\\
v_3&v_1&v_1&v_2\\
v_1&v_2&v_2&v_1\\
v_2&v_1&v_3&v_3\\
v_1&v_3&v_1&v_2\\
v_3&v_2&v_2&v_3\\
v_2&v_1&v_1&v_1\\
v_1&v_2&v_3&v_2\\
v_2&v_3&v_2&v_3\\
v_3&v_2&v_1&v_1\\
v_1&v_1&v_3&v_3\\
v_3&v_3&v_2&v_2\\
v_2&v_2&v_1&v_3\\
v_3&v_1&v_3&v_1\\
v_1&v_3&v_2&v_3
\end{bmatrix}
\begin{bmatrix}
v_3&v_2&v_1&v_2\\
v_2&v_3&v_3&v_1\\
v_1&v_1&v_2&v_2\\
v_3&v_2&v_3&v_3\\
v_1&v_3&v_1&v_1\\
v_2&v_1&v_3&v_2\\
v_1&v_2&v_2&v_3\\
v_3&v_3&v_3&v_1\\
v_2&v_1&v_1&v_3\\
v_3&v_2&v_2&v_2\\
v_1&v_1&v_3&v_1\\
v_3&v_3&v_1&v_3\\
v_2&v_1&v_2&v_2\\
v_1&v_2&v_3&v_3\\
v_3&v_3&v_2&v_1\\
v_2&v_2&v_1&v_2
\end{bmatrix}
\begin{bmatrix}
v_1&v_1&v_2&v_3\\
v_3&v_2&v_3&v_1\\
v_2&v_3&v_2&v_2\\
v_1&v_2&v_1&v_3\\
v_3&v_1&v_3&v_2\\
v_1&v_3&v_2&v_1\\
v_2&v_2&v_3&v_3\\
v_3&v_1&v_1&v_1\\
v_1&v_3&v_3&v_2\\
v_2&v_2&v_2&v_1\\
v_3&v_1&v_3&v_3\\
v_2&v_3&v_1&v_2\\
v_1&v_2&v_3&v_1\\
v_2&v_1&v_2&v_3\\
v_3&v_3&v_1&v_1\\
v_1&v_2&v_2&v_2
\end{bmatrix}
\begin{bmatrix}
v_2&v_1&v_3&v_1\\
v_3&v_3&v_2&v_3\\
v_1&v_2&v_1&v_1\\
v_2&v_3&v_3&v_2\\
v_3&v_1&v_2&v_1\\
v_1&v_3&v_1&v_3\\
v_3&v_2&v_3&v_2\\
v_2&v_3&v_2&v_1\\
v_1&v_1&v_1&v_2\\
v_3&v_3&v_3&v_3\\
v_2&v_2&v_1&v_1\\
v_1&v_3&v_2&v_2\\
v_3&v_1&v_1&v_3\\
v_2&v_2&v_3&v_2\\
v_1&v_1&v_2&v_1\\
v_3&v_3&v_1&v_2
\end{bmatrix}
\begin{bmatrix}
v_2&v_2&v_2&v_3\\
v_1&v_3&v_3&v_1\\
v_2&v_1&v_1&v_2\\
v_3&v_2&v_2&v_1\\
v_2&v_3&v_3&v_3\\
v_1&v_2&v_1&v_2\\
v_3&v_1&v_2&v_3\\
v_2&v_3&v_1&v_1\\
v_1&v_1&v_3&v_2\\
v_3&v_2&v_1&v_3\\
v_2&v_1&v_2&v_1\\
v_3&v_3&v_3&v_2\\
v_1&v_1&v_1&v_3\\
v_2&v_2&v_3&v_1\\
v_3&v_1&v_2&v_2\\
v_2&v_3&v_1&v_3
\end{bmatrix}$$

\section{Future Work}
This paper has constructed a framework that can be used to generate orderings for Hamming graphs in such a way that they must satisfy the radio condition. However, there has yet to be found any reasonable method within this framework of weak generators that ensures avoiding repetition i.e. an ordering generated in this way will only satisfy the conditions for a (consecutive) radio labeling up to labeling vertices multiple times. If a clever constraint on $O'$ were found that entailed repetition, then this would provide us with a method for generating radio labelings relatively quickly by simply searching for $O'$ satisfying the constraints of Theorem \ref{LambdasGeneral} and repetition. Additionally, if there was a construction for a potential radio labeling of a Hamming graph, one need only prove that this construction satisfies Theorem \ref{LambdasGeneral} and does not create repetition.

\begin{question} Is there an efficient process for generating $O'$ without creating repetition?
\end{question}

This paper has also taken the first step toward filling the gap in Figure \ref{Fig:gap} but it is still unknown if there is a tighter upper bound than the one presented as Corollary \ref{upperbound}.

\begin{question} Is Corollary \ref{upperbound} a tight upper bound?
\end{question}

\begin{question} Is Theorem \ref{upperboundgeneral} a tight upper bound?
\end{question}

\begin{question} In light of both Corollary \ref{BoundaryForced} and Corollary \ref{Lambdas}, as well as the process illustrated for labeling $K_3^4$, is there a method for generating valid $O'$ for $K_n^t$ where $t=\frac{n(n^2-1)}{6}$?
\end{question}

\begin{question} Are there choices of $\Delta_n^i$ that yield $\Lambda_s^i$ with any significant or useful algebraic structure?
\end{question}

\begin{question} Is there a method to algebraically study consecutive radio labelings of Hamming graphs?
\end{question}

\bibliographystyle{amsplain}
\bibliography{radio}

\end{document}